\documentclass[english,12pt]{scrartcl}
\usepackage[T1]{fontenc}
\usepackage[latin9]{inputenc}
\usepackage{a4wide}%{geometry}
\usepackage{babel}
\usepackage{prettyref}
\usepackage{mathtools}
\usepackage{amsmath}
\usepackage{amsthm}
\usepackage{amssymb}
\usepackage[unicode=true,pdfusetitle,
 bookmarks=true,bookmarksnumbered=false,bookmarksopen=false,
 breaklinks=false,pdfborder={0 0 1},backref=false,colorlinks=false]
 {hyperref}
\usepackage{color}

\makeatletter
\theoremstyle{plain}
\newtheorem{thm}{\protect\theoremname}[section]
\theoremstyle{definition}
\newtheorem{problem}[thm]{\protect\problemname}
\theoremstyle{plain}

\theoremstyle{remark}
\newtheorem{rem}[thm]{\protect\remarkname}
\theoremstyle{plain}
\newtheorem{prop}[thm]{\protect\propositionname}
\theoremstyle{plain}
\newtheorem{cor}[thm]{\protect\corollaryname}
\theoremstyle{definition}

\theoremstyle{definition}

\@ifundefined{date}{}{\date{}}
\usepackage{prettyref}
\usepackage{dsfont}
\usepackage{enumerate}
\allowdisplaybreaks

\newcommand*{\e}{\mathrm{e}}
\renewcommand*{\i}{\mathrm{i}}
\newcommand{\m}{\operatorname{m}}

\newcommand{\cow}{\textnormal{co-}\tau_{\textnormal{w}}}

\newcommand{\R}{\mathbb{R}}
\newcommand{\1}{\mathds{1}}
\newcommand{\N}{\mathbb{N}}
\newcommand{\CC}{\mathbb{C}}
\newcommand{\dom}{\operatorname{dom}}
\newcommand{\ran}{\operatorname{ran}}

\renewcommand{\d}{\,\mathrm{d}}
\renewcommand{\Re}{\operatorname{Re}}

\renewcommand{\tilde}{\widetilde}

\newcommand{\sh}{s_{\textnormal{h}}}
\newcommand{\sbb}{s_{\textnormal{b}}}

\newrefformat{subsec}{Subsection \ref{#1}}
\newrefformat{prob}{Problem \ref{#1}}
\newrefformat{prop}{Proposition \ref{#1}}
\newrefformat{lem}{Lemma \ref{#1}}
\newrefformat{thm}{Theorem \ref{#1}}
\newrefformat{cor}{Corollary \ref{#1}}
\newrefformat{rem}{Remark \ref{#1}}
\newrefformat{exa}{Example \ref{#1}}
\newrefformat{sub}{Subsection \ref{#1}}
\newrefformat{eq}{(\ref{#1})}

\theoremstyle{definition}

\newrefformat{hyp}{Hypotheses \ref{#1}}

\makeatother

\providecommand{\corollaryname}{Corollary}
\providecommand{\definitionname}{Definition}
\providecommand{\examplename}{Example}
\providecommand{\lemmaname}{Lemma}
\providecommand{\problemname}{Problem}
\providecommand{\propositionname}{Proposition}
\providecommand{\remarkname}{Remark}
\providecommand{\theoremname}{Theorem}

\begin{document}
\title{Evolutionary Equations are $G$-compact}
\author{Kre\v{s}imir Burazin\thanks{School of Applied Mathematics and Computer Science, J. J. Strossmayer University of Osijek, Croatia},\ Marko Erceg\thanks{Department of Mathematics, Faculty of Science, University of Zagreb, Croatia},\ and
Marcus Waurick\thanks{Institut f\"ur Angewandte Analysis, TU Bergakademie Freiberg, Germany}}
\maketitle
\begin{abstract}\textbf{Abstract} We prove a compactness result related to $G$-convergence for autonomous evolutionary equations in the sense of Picard. Compared to previous work related to applications, we do not require any boundedness or regularity of the underlying spatial domain; nor do we assume any periodicity or ergodicity assumption on the potentially oscillatory part. In terms of abstract evolutionary equations, we remove any compactness assumptions of the resolvent modulo kernel of the spatial operator. 
To achieve the results, we introduced a slightly more general class of material laws. 
As a by-product, we also provide a criterion for $G$-convergence for time-dependent equations solely in terms of static equations. 
\end{abstract}
\label{=00005Cnoindent}\textbf{MSC2020 } Primary 35B27 
Secondary 78M40, 80M40, 47D06  

\tableofcontents{}
\newpage
\section{Introduction}

The theory of evolutionary equations was initiated by the seminal paper \cite{Picard2009}. It comprises of a space-time Hilbert space framework for (predominantly) time-dependent partial differential equations. The restriction to the Hilbert space case and the particular focus on non-homogeneous right-hand sides rather than on initial value problems led to a fairly wide class of examples. This includes for instance mixed type problems with rough interfaces other more traditional approaches like semi-group theory failed to cover. In a nutshell, evolutionary equations are operator equations of the form
\[
   (\partial_t \mathcal{M} +A)u=f,
\] set in a suitable space-time Hilbert space to be introduced in detail below. For the time being it is enough to think of $\partial_t$ as an operator realisation of the time-derivative, $A$ a skew-selfadjoint operator (usually comprising of first order spatial derivatives) in a (separable) spatial Hilbert space $H$ canonically extended to space-time, and $\mathcal{M}$ to be a bounded linear operator in space-time assumed to be causal and invariant under time-shifts. As a consequence, see \cite[Theorem 8.2.1]{STW22} and Theorem \ref{thm:repres} below, there exists a unique bounded, holomorphic operator-valued function $M\colon \CC_{\Re>\nu} \to L(H)$ defined on a half space $\CC_{\Re>\nu}\coloneqq \{z\in \CC; \Re z>\nu\}$ for some $\nu>0$ such that
\[
   \mathcal{M} = M(\partial_t),
\] where the latter is meant in a sense of an explicit functional calculus by representing the above introduced time-derivative as a multiplication operator. In applications, the constitutive relations are encoded in $\mathcal{M}$ (and thus in  $M$); therefore $M$ is also called \textbf{material law} and $\mathcal{M}$ (the associated) \textbf{material law operator}. The versatility of evolutionary equations can in part also be explained by its easy solution criterion, which is to show that there exist $\rho \geq \nu$ and $\alpha>0$ such that
\begin{equation}\label{eq:WP}
    \Re \langle \phi,zM(z) \phi\rangle_{H}\geq \alpha\|\phi\|_{H}^2\quad(z\in \CC_{\Re>\rho}, \phi\in H).
\end{equation}
Given the latter condition, $\overline{(\partial_t \mathcal{M} +A)}$ is continuously invertible with norm of the inverse bounded by $1/\alpha$.

Since 2009 a great deal of research has been devoted to evolutionary equations and, among other things, generalisations to non-autonomous and non-linear cases were studied and the asymptotic behaviour of solutions was addressed, see the monographs \cite{PM11,PMTW20} for the theory mainly focussing on the presentation of various examples and the survey paper \cite{PTW15_WP_P} as well as the recent book \cite{STW22} for a wrap up not only of examples but also of the various other aspects of the theory. Among these is the study of homogenisation problems associated to the theory of evolutionary equations, which started with the PhD Thesis \cite{WaurickPhD} and has been further developed since then. Within the theory of evolutionary equations, (autonomous) homogenisation problems are of the form
\[
     (\partial_t {M}_n(\partial_t) +A)u_n=f,
\]
for a sequence of holomorphic functions $(M_n)_n$ defined on a common half space. Following Spagnolo \cite{Spagnolo1967,Spagnolo1976} and given a skew-selfadjoint operator $A$ on $H$, we say that a (locally bounded) sequence of material laws $(M_n)_n$ satisfying \eqref{eq:WP}, $\alpha$ being fixed throughout the paper, \textbf{$G$-converges} to a (holomorphic, operator-valued) $M$, if $ M$ satisfies \eqref{eq:WP} and 
\[\overline{(\partial_t {M}_{n}(\partial_t) +A)}^{-1}\to \overline{(\partial_t {M}(\partial_t) +A)}^{-1}\]
in the weak operator topology.
The question of $G$-compactness in the present setting can now be written as follows.
\begin{problem} For any \textit{bounded} sequence of material laws $(M_n)_n$ satisfying \eqref{eq:WP} can we find a subsequence $(M_{\pi(n)})_n$ and a material law $M$ satisfying \eqref{eq:WP} such that $(M_{\pi(n)})_n$ $G$-converges to $M$?
\end{problem}
In several research contributions, this problem has been addressed. We refer to the references \cite{W12_HO,W14_G} for situations with $A=0$, to \cite{W14_FE,W13_HP} for $A\neq 0$ with compact resolvent and to \cite{W16_HPDE} for a hybrid of these two cases; see also the example concerning Maxwell's equations in \cite{W18_NHC} and \cite[Sections 5 and 6]{WaurickOv}; we particularly refer to \cite{BSW23} for the most general situation up-to-date. In these references the overarching strategy always is the same: firstly identify suitable criteria for the convergence of $(M_n)_n$ and for the operator $A$ to have $\overline{(\partial_t {M}_{n}(\partial_t) +A)}^{-1}\to \overline{(\partial_t {M}(\partial_t) +A)}^{-1}$ in the weak operator topology and, then secondly, use an operator-valued analogue of Montel's theorem (see Theorem \ref{thm:holcomp} below) in order to guarantee the satisfaction of the derived criteria for a subsequence of $(M_n)_n$. This strategy of proof seemingly necessitates conditions on $A$. More precisely, all of the results available in the literature share that $A$ is assumed to have compact resolvent modulo kernel, that is,
\begin{equation}\label{eq:Acomp}
    \dom(A)\cap \ker(A)^\bot \hookrightarrow\hookrightarrow H.
\end{equation}
In some cases \cite{W16_HPDE} additional structural assumptions on the  sequence $(M_n)_n$ were necessary to obtain the satisfaction of \eqref{eq:WP} also for the $G$-limit $M$. Altogether, the interplay of conditions for $(M_n)_n$ and $A$ is rather involved and the respective proofs may be considered quite technical with demanding computations, see \cite{W16_HPDE}, or necessitate the involvement of more (abstract) theory, see \cite{BSW23}. %Altogether, the interplay of conditions for $(M_n)_n$ and $A$ is rather involved and the respective proofs may be considered rather technical. Note that as long as \cite{W16_HPDE} is concerned demanding computations are required lacking both structural insights as well as objective appeal of the results themselves.
In any case, the techniques applied do not work without condition \eqref{eq:Acomp}.

It is the aim of the present contribution to entirely remove these intricacies and subtleties and to completely change the point of view: Instead of finding conditions for $(M_n)_n$ (and $A$) to obtain a continuous dependence result (see also \cite{W16_H,BSW23}), we take a step back and \emph{assume} weak operator topology convergence of the solution operators $\overline{(\partial_t {M}_{n}(\partial_t) +A)}^{-1}$ to some space-time operator $B$ right away and then, from there, try to derive the existence of a material law $M$ satisfying \eqref{eq:WP} such that $B=\overline{(\partial_t {M}(\partial_t) +A)}^{-1}$.  The $G$-compactness statement can then be obtained by applying sequential compactness of bounded sets of bounded linear operators defined on separable Hilbert spaces. For this strategy to work we enlarge the usually considered class of holomorphic functions $M$ by allowing for $M$ to be unbounded at $\infty$, so that $M$ is then a \textbf{generalised material law}; see the next section for the precise definition. The $G$-compactness statement will be obtained within this bigger class.

The strategy of proof is based on the proof of a similar result for so-called Friedrichs systems in a Hilbert space setting, see \cite{BEW23, Burazin14}. These results are conveniently formulated for predominantly static partial differential equations. The abstract nature of the result in \cite{BEW23}, however, also allows for time-dependence.

Finally, the main contribution of the present article, $G$-compactness for evolutionary equations, reads as follows; recall $\rho\geq \nu>0$ (see also Theorem \ref{thm:seqcomp}).

\begin{thm}\label{thm:mainintro} Let $(M_n)_n$ be a locally bounded sequence of generalised material laws all satisfying \eqref{eq:WP} and let $A$ be skew-selfadjoint. Then there exist a subsequence $(M_{\pi(n)})_n$ and a generalised material law $M$ satisfying \eqref{eq:WP} such that, in the weak operator topology,
	\[
	\overline{(\partial_t {M}_{\pi(n)}(\partial_t) +A)}^{-1}\to \overline{(\partial_t {M}(\partial_t) +A)}^{-1};
	\]that is, $(M_{\pi(n)})_n$ $G$-converges to $M$.
\end{thm}

Note that any bounded sequence $(M_n)_n$ of material laws is also a locally bounded sequence of generalised material laws (see Section \ref{sec:prl}) so that Theorem \ref{thm:mainintro} applies to all the cases discussed in the literature so far.

We briefly comment on the consequences of the removal of the additional conditions on $A$ and $(M_n)_n$ that were previously used in the literature. To start off with, we generalise all $G$-compactness statements available for evolutionary equation simultaneously. The very technical proof of a condition similar to \eqref{eq:WP} for the limit material law in \cite{W16_HPDE} can be omitted and the present result can be used instead. The involved conditions for $A$ can also be lifted in this reference. Even in the case of $A$ satisfying condition \eqref{eq:Acomp}, the material laws are not required to \emph{a priori} converge in the so-called parametrised Schur topology (see \cite{BSW23} for a corresponding definition) in order to apply a corollary of one of the main results in \cite{BSW23}. In fact, a combination of the results presented here and \cite{BSW23} even \emph{implies} convergence of the material laws in the parametrised Schur topology, if they $G$-converged in the first place.

For applications, the removal of conditions on $A$ implies that the main theorem of the present contribution now applies similarly to \emph{all} open subsets of $\R^d$ as underlying spatial domain and no regularity conditions are needed anymore. Moreover, we shall see that it suffices to consider time-independent(!) problems to deduce $G$-convergence for time-dependent problems on arbitrary underlying spatial domains. The missing structural assumptions on $(M_n)_n$ now immediately yield $G$-compactness statements also for highly oscillatory mixed type equations for high-dimensional equations.

In passing of the proof of $G$-compactness, we also provide a criterion for $G$-convergence, which is the second main contribution of the manuscript at hand. The result can be shortened to the following (see also Theorem \ref{thm:mainchar}).

\begin{thm}\label{thm:mainintro2} Let $(M_n)_n$ be a locally bounded sequence of generalised material laws all satisfying \eqref{eq:WP}. Then the following conditions are equivalent:
	\begin{enumerate}
		\item[(i)] There exists a generalised material law $M$ with $(M_n)_n$ $G$-converging to $M$.
		\item[(ii)] For all sufficiently large $\mu>0$ and all $\psi \in H$, the sequence $(\phi_{n})_n$ in $\dom(A)$ given by
		\[
		\mu M_n (\mu) \phi_{n}+A\phi_{n} = \psi
		\]
		is weakly convergent.
	\end{enumerate}
\end{thm}

We sketch out the plan of the paper. The next section summarises the functional analytic framework of evolutionary equations and provides the precise definitions and notions needed in the following. In Section \ref{sec:Gcom} we state and prove the main results of the paper and in Section \ref{sec:app} we apply our finding to transport equations with highly oscillatory coefficients. We conclude with Section \ref{sec:concl} adding some closing remarks.

\section{Preliminaries -- Evolutionary Equations}\label{sec:prl}

This section is devoted to briefly summarise the notion of evolutionary equations. For the details, we shall frequently refer to the recent monograph \cite{STW22}. 

Throughout this section, let $H$ be a Hilbert space. Let $\rho\in \R$. Then we define $L_{2,\rho}(\R;H)$, an exponentially weighted $H$-valued $L_2$-space, as follows
\[
   L_{2,\rho}(\R;H)\coloneqq \{ f\in L_{1,\textnormal{loc}}(\R;H); \int_{\R}\|f(s)\|_H^2 \e^{-2\rho s} \d  s<\infty\}
\]
with the obvious norm and scalar product. It is not difficult to see that $L_{2,\rho}(\R;H)$ actually is a Hilbert space, see \cite[Proposition 3.1.4 and p.~42]{STW22}.

We define
\[
  \partial_t \colon H_\rho^1(\R;H) \subseteq L_{2,\rho}(\R;H)\to L_{2,\rho}(\R;H), \ f\overset{\partial_t}{\mapsto} f'
\]the (weak) derivative with maximal domain 
\[H_\rho^1(\R;H)\coloneqq \{f\in L_{2,\rho}(\R;H); f'\in L_{2,\rho}(\R;H)\}.
\]
For convenience, we shall often refrain from writing an additional $\rho$ in the index of $\partial_t$ as the particular value of $\rho$ is mostly clear from the context (or, if any  fixed $\rho$ will do). If the particuar value of $\rho$ is of some concern, we shall however also write $\partial_{t,\rho}$.
In any case, for all $\rho\neq 0$, $\partial_{t,\rho}$ is continuously invertible; for $\rho>0$, we have
\[
   \partial_t^{-1} f (s) = \int_{-\infty}^s f(r) \d  r\quad (f\in L_{2,\rho}(\R;H)),
\]see \cite[p.~43 and Proposition 4.1.1]{STW22}. 

The so-defined derivative operator has an explicit spectral representation, which can be found using the Fourier--Laplace transformation $\mathcal{L}_\rho\colon L_{2,\rho}(\R;H)\to L_2(\R;H)$. This operator is the \emph{unitary} extension of the mapping assigning to each continuous function with compact support $f\in C_c(\R;H)$ its Fourier--Laplace transform $\mathcal{L}_\rho f$ given by
\[
   \mathcal{L}_\rho f (t) = \frac{1}{\sqrt{2\pi}} \int_{\R} \e^{-(\i  t + \rho)s} f(s)  \d  s;
\]see also \cite[Remark 5.2.1]{STW22}. Introducing 
\[\m \colon \dom(\m)\subseteq L_{2}(\R;H)\to L_{2}(\R;H), \ f\overset{\m}{\mapsto} (t\mapsto tf(t))
\]
with maximal domain, we find
\[
    \partial_t = \mathcal{L}_\rho^* (\i  \m +\rho) \mathcal{L}_\rho
\]
for all $\rho\in \R$; see \cite[Theorem 5.2.3]{STW22}. The latter spectral representation yields a means to define operator-valued functions of $\partial_t$. As it was highlighted in the introduction, already autonomous and causal operators dictate the consideration of holomorphic (operator-valued) functions of $\partial_t$.

We say that $M$ is a \textbf{generalised material law}, if $M\colon \dom(M)\subseteq \CC\to L(H)$ is holomorphic and there exists $\rho\in \R$ such that $\CC_{\Re > \rho}\subseteq \dom(M)$. The infimum over all such $\rho\in \R$ is called \textbf{abscissa of holomorphicity (of $M$)} and denoted by $\sh(M)$. A generalised material law is called \textbf{material law}, if, in addition, there exists $\rho>\sh(M)$ such that
\[
\|M\|_{\rho,\infty}\coloneqq   \sup_{z\in \CC_{\Re > \rho}}\|M(z)\|<\infty,
\]
i.e.~if $M$ is \textbf{bounded} on some $\CC_{\Re > \rho}$.
The infimum over all $\rho$ with the said properties defines the \textbf{abscissa of boundedness (of $M$)} and will be denoted by $\sbb(M)$, see \cite[Section 5.3]{STW22} for details. For a generalised material law $M$, we define the associated \textbf{material law operator} $M(\partial_t)$ for $\rho>\sh(M)$ via
\[
M( \partial_{t,\rho}) = \mathcal{L}_\rho^* M(\i  \m +\rho) \mathcal{L}_\rho,
\]
where $M(\i  \m +\rho)$ defines a (potentially unbounded) multiplication operator endowed with maximal domain uniquely determined by
\[
\big(M(\i  \m +\rho)\phi\big)(\xi) \coloneqq M(\i  \xi +\rho)\phi(\xi)\quad ((\hbox{a.e.}) \,\xi\in \R, \phi \in C_c(\R;H)).
\]For a bounded material law $M$, the material law operator $ M( \partial_t)$ is independent of the particular choice of $\rho$ in the sense that \[\mathcal{L}_\rho^* M(\i  \m +\rho) \mathcal{L}_\rho f = \mathcal{L}_\mu^* M(\i  \m +\mu) \mathcal{L}_\mu f\] as long as $f\in L_{2,\rho}(\R;H)\cap L_{2,\mu}(\R;H)$ for $\rho,\mu>\sbb(M)$, see \cite[Theorem 5.3.6]{STW22}. For later purposes we introduce the spaces
\begin{align*}
\mathcal{M}_g(H,\nu) & \coloneqq \{ M \text{ generalised material law}; \nu\geq \sh(M)\},\text{ and}\\
\mathcal{M}(H,\nu) & \coloneqq \{ M \text{ material law}; \nu\geq \sbb(M)\}.
\end{align*}
Let us note that then the above introduced mapping $\|\cdot\|_{\nu,\infty}$
represents a norm on $\mathcal{M}(H,\nu)$.
The fundamental theorem forming the foundation of the theory of evolutionary equations is Picard's theorem, which we will provide next in a slightly more general form. For this we remark that for a closed and densely defined linear operator $A$ on $H$, we denote its canonical extension to $L_{2,\rho}(\R;H)$ by the same name; the domain is then $L_{2,\rho}(\R;\dom(A))$. Note that the canonical extension inherits the properties of being closed and densely defined. Moreover, if $A$ is skew-selfadjoint on $H$, then so is its extension; see \cite[Exercise 6.1]{STW22}. A bounded linear operator $\mathcal{S}$ on $L_{2,\rho}(\R;H)$ is called \textbf{causal}, if for all $a\in \R$ and $f\in L_{2,\rho}(\R;H)$ we have
\[
\1_{(-\infty,a]} \mathcal{S}    \1_{(-\infty,a]} f=    \1_{(-\infty,a]} \mathcal{S}f.
\]
\begin{thm}[{{Picard's Theorem, see e.g.~\cite[Theorem 6.2.1 and Remark 6.3.3]{STW22}}}] Let $M$ be a generalised material law and assume there exist $\mu\geq \sh(M)$ and $\alpha>0$ such that
	\[
	\Re \langle \phi,zM(z)\phi\rangle\geq \alpha\|\phi\|^2\quad (\phi\in H,z\in \CC_{\Re>\mu}).
	\]
	Let $A\colon \dom(A)\subseteq H\to H$ be skew-selfadjoint.
	
	Then for $\rho>\mu$  the operator $B\coloneqq \partial_t M(\partial_t)+A$ with domain $\dom(\partial_{t}M(\partial_t))\cap L_{2,\rho}(\R;\dom(A))$ is closable in $L_{2,\rho}(\R;H)$. The closure of $B$ is continuously invertible, $\mathcal{S}\coloneqq \overline{B}^{-1}$ with $\|\mathcal{S}\|\leq 1/\alpha$ and $\mathcal{S}$ is causal.
\end{thm}
\begin{rem}\label{rem:holo}
	In Picard's Theorem, by composition, the mapping 
	\[
	N\colon   z\mapsto (zM(z)+A)^{-1}
	\]
	is $L(H)$-valued, bounded on $\CC_{\Re\geq\mu}$, and holomorphic on $\CC_{\Re>\rho}$ for all $\rho>\mu$. Hence, $N$ itself is a material law with $\sbb(N)\leq \mu$. It is part of the proof of Picard's Theorem in \cite{STW22} to show that actually
	\[
	\mathcal{S} = N(\partial_t)
	\]
	holds. For this note that the proof provided in \cite{STW22} does not require $M$ to be bounded (cf.~\cite[Remark 6.3.3]{STW22}).
\end{rem}

For later purposes, we also recall a representation theorem for causal and autonomous operators on $L_{2,\rho}(\R;H)$. A bounded linear operator $\mathcal{N}$ on $L_{2,\rho}(\R;H)$ is called \textbf{autonomous}, if for $h\in \R$, $\tau_h \mathcal{N} = \mathcal{N}\tau_h$, where $\tau_h f(t)\coloneqq f(t+h)$.

\begin{thm}[{{\cite[Theorem 8.2.1]{STW22}, see also \cite{FS55}}}]\label{thm:repres} Let $\nu\in \R$ and $\mathcal{N}\in L(L_{2,\nu}(\R;H))$ autonomous and causal. Then for all $\rho>\nu$, $\mathcal{N}|_{L_{2,\nu}\cap L_{2,\rho}}$ admits a unique continuous extension $\mathcal{N}_\rho$ to $L_{2,\rho}(\R;H)$ and there exists a unique holomorphic and bounded $N\colon \CC_{\Re>\nu}\to L(H)$ such that for all $\rho>\nu$
\[
   \mathcal{N}_\rho=N(\partial_{t,\rho}).
   \]
\end{thm}
\begin{rem}\label{rem:convfoures} Note that the converse of Theorem \ref{thm:repres} is also true: Every holomorphic and bounded $N\colon \CC_{\Re>\nu}\to L(H)$ defines via $N(\partial_{t,\rho})$ a causal and autonomous operator on $L_{2,\rho}(\R;H)$, the restriction of which to $L_{2,\rho}\cap L_{2,\nu}$ admits a (unique) continuous, causal and autonomous extension to $L_{2,\nu}(\R;H)$. For the proof note that the operator in question being autonomous is easy by observing that $\tau_h$ is multiplication by an exponential in the Fourier--Laplace transformed side. Causality follows with the help of \cite[Theorem 5.3.6]{STW22}.
\end{rem}

Having had set the stage of evolutionary equations, we may now move on to the main body of the present manuscript.

\section{A $G$-compactness Theorem for Evolutionary Equations}\label{sec:Gcom}

Throughout this section, we let $H$ be a separable Hilbert space. Furthermore, throughout this section we let $A\colon \dom(A)\subseteq H\to H$ be a skew-selfadjoint operator.
\begin{rem}\label{rem:domAsep}
Quickly recall that since $H$ is separable, so is $H\times H$. Thus, as a metric subspace $A$ (considered as a relation) is also separable. Finally, as the latter is unitarily equivalent to $\dom(A)$ endowed with the graph norm, $\dom(A)$ is separable as well. 
\end{rem}

Recalling the sets $\mathcal{M}(H,\nu)$ and $\mathcal{M}_g(H,\nu)$, we define the notion of $G$-convergence next. For this, we call a sequence $(M_n)_n$ in $\mathcal{M}_g(H,\nu)$ \textbf{locally bounded}, if for all $K\subseteq \CC_{\Re>\nu}$ compact, $\sup_{z\in K,n\in \N}\|M_n(z)\|<\infty$. Furthermore,
it is called \textbf{bounded} if the same holds uniformly in $z\in \CC_{\Re>\nu}$,
i.e.~$\sup_{z\in \CC_{\Re>\nu},n\in \N}\|M_n(z)\|=\sup_{n\in\N}\|M_n\|_{\nu,\infty}<\infty$. Of course, we then have $M_n\in\mathcal{M}(H,\nu)$, $n\in\N$. For $\alpha>0$
we introduce  the sets
\[
\mathcal{M}_{(g)}(H,\nu,\alpha)\coloneqq \Bigl\{ M\in \mathcal{M}_{(g)}(H,\nu); \;
	\Re \langle zM(z)\phi,\phi\rangle\geq \alpha\|\phi\|^2 \ \;
	(z\in\CC_{\Re>\nu}, \; \phi\in H)\Bigr\}\,.
\]
A locally bounded sequence of generalised material laws $(M_n)_n$ from $\mathcal{M}_{g}(H,\nu,\alpha)$ is said to \textbf{$G$-converge (with respect to $A$)} to some $M\in \mathcal{M}_g(H,\nu,\alpha)$, if
\[
\overline{\partial_t M_n(\partial_t)+A}^{-1} \to    \overline{\partial_t M(\partial_t)+A}^{-1}
\]in the weak operator topology of $L(L_{2,\rho}(\R;H))$ for all $\rho>\nu$.

\begin{rem}
The results gathered in \cite{WaurickOv} show that whether or not any sequence of material laws $G$-converges to some limit is heavily dependent on the considered operator $A$. In fact, it can also happen that the material laws do $G$-converge but to a limit, which depends on properties of $A$ and can differ for different $A$. As we have fixed the operator $A$ in this section we do not mention this dependence in the following to avoid unnecessary clutter in the notation. 
\end{rem}

Our main result now reads as follows (also cp.~Theorem \ref{thm:mainintro}).
\begin{thm}\label{thm:seqcomp} Let $\nu\in \R, \alpha>0$. Then, every locally bounded sequence in $\mathcal{M}_g(H,\nu,\alpha)$ has a $G$-convergent subsequence with limit in $\mathcal{M}_g(H,\nu,\alpha)$.
\end{thm}
\begin{rem}
	The results available in the literature (see again \cite{WaurickOv} for a good overview) state the existence of $G$-convergent subsequences for bounded sequences in $\mathcal{M}(H,\nu,\alpha)$ (and require additional information about $A$). On the other hand, the limit in this case will also be bounded. Thus, compared to the literature the assumptions here are weaker in that there is no additional condition on $A$ and that the material laws are allowed to diverge at infinity. The price to pay is that, even when starting from a bounded sequence of material laws, one cannot conclude boundedness of the limit generalised material law here.
	A more precise description of the results we are able to get under additional assumptions is given in the following corollary (see also Remark \ref{rem:G-conv_bdd}). 
\end{rem}

\begin{cor}\label{cor:seqcomp_bdd}
Let $\nu\in\R$, $\alpha>0$. 
If a bounded sequence $(M_n)_n$ in $\mathcal{M}(H,\nu,\alpha)$ G-converges to $M\in\mathcal{M}_g(H,\nu,\alpha)$, then
\begin{equation}\label{eq:Gcomp_linear_growth}
\|M(z)\|\leq \frac{\sup_{n\in\N}\|M_n\|_{\nu,\infty}^2}{\alpha}|z| \quad (z\in \CC_{\Re>\nu}).
\end{equation}

Furthermore, if, in addition for some $c>0$,
\begin{equation}\label{eq:Gcomp_addassump}
\|M_n(z)\| \leq \frac{c}{\sqrt{|z|}} \quad (z\in \CC_{\Re>\nu}),
\end{equation}
then $M\in\mathcal{M}(H,\nu,\alpha)$, i.e.~$\|M\|_{\nu,\infty}\leq \frac{c^2}{\alpha}$.
\end{cor}

\begin{rem}\label{rem:cor}
One can see that for bounded material laws we do not have a satisfactory result. 
Indeed, at this moment for bounded sequences in $\mathcal{M}(H,\nu,\alpha)$ we are not able to either prove or disprove that a $G$-limit is a material law, i.e.~contained in $\mathcal{M}(H,\nu,\alpha)$ (note the linear growth appearing in \eqref{eq:Gcomp_linear_growth}). On the other hand, the additional growth assumption \eqref{eq:Gcomp_addassump} of the corollary is too restrictive (e.g.~an important class of constant (in $z$) material laws is not covered, and so neither of two examples presented in Section \ref{sec:app}). Since we do not have any evidence that our result given in the corollary is optimal, an open question of minimal requirements needed to ensure boundedness of the $G$-limits is left to be addressed (see also Remark \ref{rem:G-conv_bdd}).
Our conjecture is that solely boundedness of material laws in $\mathcal{M}(H,\nu,\alpha)$ should not be sufficient, and we hope that the results of the corollary can help in constructing a suitable counterexample.  
\end{rem}

The proof of Theorem \ref{thm:seqcomp} is based on the interconnected application of several compactness statements. To start with, we provide  one of the main ingredients next, which is a result originally proved for (partial differential type) Friedrichs systems (in real Hilbert spaces) from \cite{Burazin14}, 
and recently generalised in \cite{BEW23}. Since we are interested in $A$ skew-selfadjoint only, the condition (K1) in \cite{Burazin14} is trivially satisfied and, thus, also allows for some simplifications in the proof. As the result is already contained in \cite[Corollary 2.11]{BEW23} and requires only some minor modifications of the corresponding result in \cite{Burazin14}, we only sketch its proof below. To present the respective result we introduce, for $0<\alpha<\beta$,
\[
  \mathcal{F}(\alpha,\beta,H)\coloneqq \Bigl\{C\in L(H); \Re \langle \phi,C\phi\rangle\geq \alpha\|\phi\|^2, \frac{1}{\beta}\|C\phi \|^2 \leq \Re \langle \phi,C\phi\rangle\quad(\phi\in H) \Bigr\}.
  \]
  \begin{rem}\label{rem:ca}
  Let $C\in \mathcal{F}(\alpha,\beta,H)$. Then the operator $C+A$ is continuously invertible with $(C+A)^{-1}\in L(H)$ and $(C+A)^{-1}\in L(H,\dom(A))$ with $\|(C+A)^{-1}\|_{L(H)}\leq 1/\alpha$ and $\|(C+A)^{-1}\|_{L(H,\dom(A))}\leq \frac{1+\alpha+\beta}{\alpha}$, see also \cite[Lemma 2.12]{EGW17_D2N}.
  \end{rem}
  \begin{thm}[{{see also \cite[Theorem 2.10 and Corollary 2.11]{BEW23}}}]\label{thm:compFS} Let $A$ be a skew-selfadjoint operator on $H$ and $(C_n)_n$ in $\mathcal{F}(\alpha,\beta,H)$. If 
\[
   (C_n+A)^{-1} \to B
\]in the weak operator topology of $L(H,\dom(A))$, then there exists $C\in \mathcal{F}(\alpha,\beta,H)$ such that 
\[
     B = (C+A)^{-1} \,.
\]
Moreover, for all $\psi\in H$ and $\phi_n\coloneqq (C_n+A)^{-1}\psi$ we obtain
\[
   \Re \langle C_n\phi_n,\phi_n\rangle \to \Re \langle C\phi,\phi\rangle \quad \text{and} \quad C_n\phi_n \to C\phi \text{ weakly},
\]
where $\phi\coloneqq (C+A)^{-1}\psi$.
\end{thm}
\begin{rem}\label{rem:sepseq} We recall that bounded sets of $L(H_1,H_2)$ for separable Hilbert spaces $H_1$ and $H_2$ are sequentially compact under the weak operator topology. In fact, either one proves this fact by a Banach--Alaoglu type argument and a standard proof for metrisability or one simply resorts to a diagonal procedure. 
\end{rem}
\begin{rem}\label{rem:HdomA} The assumption on $\bigl((C_n+A)^{-1}\bigr)_n$ to be convergent in the weak operator topology of $L(H,\dom(A))$ is a mere convenience for the subsequent proof. A formally stronger statement would be an analogous result only assuming weak operator topology convergence in $L(H)$. However, this implies the same for $L(H,\dom(A))$. Indeed, $\bigl((C_n+A)^{-1}\bigr)_n$ is uniformly bounded in $L(H,\dom(A))$ by Remark \ref{rem:ca}. Then we may choose by Remark \ref{rem:sepseq} a subsequence converging in the weak operator topology, by continuous embedding $\dom(A)\hookrightarrow H$, it follows that the limits in $L(H,\dom(A))$ and $L(H)$ coincide. A subsequence principle concludes convergence of $\bigl((C_n+A)^{-1}\bigr)_n$ in the weak operator topology of $L(H,\dom(A))$ without the need to choose subsequences; see also Remarks \ref{rem:domAsep} and \ref{rem:sepseq}.
\end{rem}
\begin{proof}[Proof of Theorem \ref{thm:compFS}]
 We define $K\in L(H)$ via $K\psi \coloneqq \psi - AB\psi$ for all $\psi\in H$. Let $\psi\in H$ and define $\phi_n\coloneqq (C_n+A)^{-1}\psi$. One has $A\phi_n \to A B\psi$ weakly in $H$ and, using the equations satisfied by $\psi$, $C_n\phi_n\to K\psi$ weakly in $H$. Skew-selfadjointness of $A$, weak convergence of $(C_n+A)^{-1}$ and the equations for $\psi$ yield $
   \Re \langle C_n \phi_n, \phi_n\rangle\to \Re \langle K\psi,B\psi\rangle.$ Then, the second inequality in the definition of $\mathcal{F}(\alpha,\beta,H)$ used for $C_n$ helps to show that  $B$ is one-to-one. Next, $\ran(B)\subseteq H$ dense is shown by taking $\psi\in \ran(B)^{\bot}$ and proving $\psi=0$ by injectivity of $B$ and the first inequality in the definition of $\mathcal{F}(\alpha,\beta,H)$. Finally, defining $C\colon \ran(B)\subseteq H \to H$ by $C( B\psi) \coloneqq K\psi$, we infer that $C$ is densely defined and well-defined as $B$ is one-to-one with dense range. It is then not difficult to confirm $C\in \mathcal{F}(\alpha,\beta,H)$.\end{proof}
\begin{rem}
Theorem \ref{thm:compFS} is a compactness theorem. Indeed, due to the separability of $H$, from an arbitrary sequence $(C_n)_n$ a subsequence can be chosen such that $\bigl((C_n+A)^{-1}\bigr)_n$ converges in the weak operator topology of $L(H,\dom(A))$; see Remark \ref{rem:sepseq} and for the separability of $\dom(A)$ see Remark \ref{rem:domAsep}.
\end{rem}

\begin{rem}\label{rem:uniqFSC}
The operator constructed in Theorem \ref{thm:compFS} is uniquely determined.
Indeed, let $C_1,C_2\in \mathcal{F}(\alpha,\beta,H)$ be such that $(C_1+A)^{-1}=(C_2+A)^{-1}$.  The latter is equivalent to $C_1+A=C_2+A$
on $\dom(A)$. Hence, $C_1=C_2$ on $\dom(A)$. Since $A$ is skew-selfadjoint, it is densely defined and as $C_1$ and $C_2$ are both continuous linear operators, $C_1=C_2$ on $H$.
\end{rem}

Next, we recall a compactness result for operator-valued holomorphic functions originally from \cite{WaurickPhD}. For this, we define for $\omega\subseteq \CC$ open and (separable) Hilbert spaces $H_1,H_2$
\[
\mathcal{H}(\omega;L(H_1,H_2))\coloneqq \{M\colon \omega\to L(H_1,H_2); M \text{ holomorphic}\}.
\]
In the correspondence to Section \ref{sec:prl},
a subset $\mathcal{B}\subseteq \mathcal{H}(\omega;L(H_1,H_2))$ is called \textbf{locally bounded}, if for all $K\subseteq \omega$ compact  $\sup_{M\in \mathcal{B},z\in K}\|M(z)\|<\infty$; a sequence $(M_n)_n$ is \textbf{locally bounded}, if $\{M_n; n\in \N\}$ is. We say that a sequence $(M_n)_n$ in $\mathcal{H}(\omega;L(H_1,H_2))$ converges to some $M\in \mathcal{H}(\omega;L(H_1,H_2))$ in the \textbf{compact open weak operator topology} (in $\cow$ for short), if for all $\phi\in H_2, \psi\in H_1$ we have
\[
   \langle \phi, M_n(z)\psi\rangle \to    \langle \phi, M(z)\psi\rangle
\]
uniformly in $z$ on compact subsets of $\omega$. Since the asserted convergence is a convergence of scalar-valued holomorphic functions, Vitali's Theorem (see \cite[Theorem 6.2.8]{Simon2015}) immediately applies and we obtain the following result:

\begin{prop}\label{prop:vitaliopval} Let $\omega\subseteq \CC$ be open, and $H_1,H_2$ be Hilbert spaces. Let $(M_n)_n$ be a locally bounded sequence in $\mathcal{H}(\omega;L(H_1,H_2))$ and let $M\in \mathcal{H}(\omega;L(H_1,H_2))$. Then the following conditions are equivalent:
\begin{enumerate}
\item[(i)] $M_n\to M$ in the compact open weak operator topology;
\item[(ii)] for all $z\in \omega$, $M_n(z)\to M(z)$ in the weak operator topology;
\end{enumerate}
and, if $\omega$ is connected, (i) and (ii) are further equivalent to
\item[(iii)] there exists a sequence $(z_k)_k$ in $\omega$ with an accumulation point in $\omega$ such that $M_n(z_k)\to M(z_k)$ in the weak operator topology for all $k\in \N$.
\end{prop}

The remarkable property of convergence in the compact open weak operator topology is a Montel type result, that is, the (sequential) compactness of locally bounded sets. An analogous result also holds for non-separable Hilbert spaces; we shall however stick to the separable case in order to avoid too much additional notions from topology.

\begin{thm}[{{\cite[Theorem 3.4]{W12_HO} and \cite{WaurickPhD}; see also \cite[Corollary 4.7]{BSW23}}}]\label{thm:holcomp} Let $H_1,H_2$ be separable Hilbert spaces, $\omega\subseteq \CC$ open. Let $(M_n)_n$ be a locally bounded sequence in $\mathcal{H}(\omega;L(H_1,H_2))$. Then $(M_n)_n$ contains a $\cow$-converging subsequence.
\end{thm}

\begin{cor}\label{cor:comp}Let $\omega\subseteq \CC$ open, $H_1,H_2$ separable Hilbert spaces. Let $(M_n)_n$ be a locally bounded sequence in $\mathcal{H}(\omega;L(H_1,H_2))$ and $M\colon \omega \to L(H_1,H_2)$.

If, for all $z\in \omega$, $M_n(z)\to M(z)$ in the weak operator topology, then $M$ is holomorphic.
\end{cor}
\begin{proof}
The claim follows with a subsequence principle using the compactness statement in Theorem \ref{thm:holcomp} together with Proposition \ref{prop:vitaliopval}.
\end{proof}

A last preparatory result for the second main theorem of this section is the following continuity statement, which essentially follows from Lebesgue's dominated convergence theorem in the Fourier--Laplace transformed side.

\begin{thm}[{{\cite[Lemma 3.5]{W12_HO}}}]\label{thm:continuity} Let $\nu\in \R$ and $(M_n)_n$ be a bounded sequence in $\mathcal{M}(H,\nu)$ and $M\in \mathcal{M}(H,\nu)$. If $M_n\stackrel{\cow}\to M$, then, $\|M(z)\|\leq\liminf_{n\to\infty}\|M_n(z)\|$ for all $z\in \CC_{\Re>\nu}$ and, for all $\rho>\nu$,  $M_n(\partial_t)\to M(\partial_t)$ in the weak operator topology of $L(L_{2,\rho}(\R;H))$. 
\end{thm}

The decisive step for the proof of Theorem \ref{thm:seqcomp} is the following characterisation result, which also proves useful for applications; cp.~Theorem \ref{thm:mainintro2}. 

\begin{thm}\label{thm:mainchar} Let $\nu>0, \alpha>0$ and $(M_n)_n$ be a locally bounded sequence in $\mathcal{M}_g(H,\nu,\alpha)$. Then the following conditions are equivalent:
	\begin{enumerate}
		\item[(i)] There exists $M\in \mathcal{M}_g(H,\nu,\alpha)$ such that $(M_n)_n$ $G$-converges to $M$ with respect to $A$.
 \item[(ii)] For all $\rho>\nu$, there exists $\mathcal{T}_\rho\in L(L_{2,\rho}(\R;H))$ such that $\overline{\partial_tM_n(\partial_t)+A}^{-1}\to \mathcal{T}_\rho$ in the weak operator topology of $L(L_{2,\rho}(\R;H))$.
 \item[(iii)] For all $\rho>\nu$ and $\psi\in H$, the sequence $(\phi_n)_n$ given by\[
     (\rho M_n(\rho)+A)\phi_n =\psi
 \]
 is weakly convergent. 
\end{enumerate}  
If either of the above is satisfied, $M$ from (i) is uniquely determined via
\[
    (\rho M_n(\rho)+A)^{-1} \to     (\rho M(\rho)+A)^{-1}
\]for all $\rho>\nu$ in the weak operator topology of $L(H)$.
\end{thm}
\begin{proof}
(i)$\Rightarrow$(ii): The claim follows by setting $\mathcal{T}_\rho \coloneqq \overline{\partial_{t,\rho}M(\partial_{t,\rho})+A}^{-1}$ for $\rho>\nu$.

(ii)$\Rightarrow$(iii): By Picard's Theorem (and its proof cp.~Remark \ref{rem:holo}), $(N_n)_n$ given by
\[
  N_n\colon  z\mapsto (zM_n(z)+A)^{-1}
\] defines a bounded sequence in $\mathcal{M}(H,\nu)$ and we have $\overline{\partial_{t}M_n(\partial_t)+A}^{-1} = N_n(\partial_t)$ on $L_{2,\rho}(\R;H)$ for $\rho>\nu$. By Remark \ref{rem:convfoures}, we obtain that $N_n(\partial_t)$ is autonomous and causal. As a limit in the weak operator topology, it is not difficult to see that $\mathcal{T}_\rho$ is, thus, also autonomous and causal for all $\rho>\nu$. Hence, we find a material law $T\colon \CC_{\Re>\nu}\to L(H)$ such that $\mathcal{T}_\rho = T(\partial_{t,\rho})$ for all $\rho>\nu$. By Theorem \ref{thm:holcomp} we find a subsequence of $(N_n)_n$ (we shall re-use the index $n$ for this subsequence) and $\tilde{T}\in \mathcal{H}(\CC_{\Re>\nu};L(H))$ with $\|\tilde{T}(z)\|\leq 1/\alpha$ for all $z\in \CC_{\Re>\nu}$ such that $N_n\to \tilde{T}$ in the compact open weak operator topology. Theorem \ref{thm:continuity} yields that $N_n(\partial_{t,\rho})\to \tilde{T}(\partial_{t,\rho})$ in the weak operator topology. By uniqueness of limits in the weak operator topology, we obtain that $T(\partial_{t,\rho})=\tilde{T}(\partial_{t,\rho})$ and, by the uniqueness statement in the representation Theorem \ref{thm:repres}, we deduce that $T=\tilde{T}$ as holomorphic mappings. A subsequence principle now leads to the whole sequence $(N_n)_n$ converging to $T$ in the compact open weak operator topology. Proposition \ref{prop:vitaliopval} ((i)$\Rightarrow$(ii)), now implies (iii).

(iii)$\Rightarrow$(i) By Theorem \ref{thm:holcomp}, we find $(N_{\pi(n)})_n$, a subsequence of $N_n\colon z\mapsto (zM_n(z)+A)^{-1} $, in $\mathcal{H}(\CC_{\Re>\nu};L(H))$ converging in the compact open weak operator topology. Denote its limit by $N$. Next, we uniquely identify this limit to deduce that $N_n$ itself converges to $N$. For this, note that by connectedness of $\CC_{\Re>\nu}$ and Proposition \ref{prop:vitaliopval}, $N(\rho)\psi$ is the weak limit of $(\phi_{\pi(n)})_n$ given by $\rho M_{\pi(n)}(\rho)\phi_{\pi(n)}+A\phi_{\pi(n)} =\psi$. Since $N$ is holomorphic, it is uniquely determined by its values on $(\nu,\infty)$ by the identity theorem and the connectedness of $\CC_{\Re>\nu}$.  As the whole sequence $(\phi_n)_n$ converges to $N(\rho)\psi$, a subsequence principle concludes that $N_n$ indeed converges to $N$ in the compact open weak operator topology. (The details are: assume that $(N_n)_n$ contains a subsequence not converging to $N$. Then we find a further subsequence converging to some $\tilde{N}$. By what we have seen above, $N$ and $\tilde{N}$ coincide on $(\nu,\infty)$; which, by the identity theorem implies $N=\tilde{N}$, which is a contradiction.)
Applying Theorem \ref{thm:continuity} we get that, for any $\rho>\nu$,
$\overline{\partial_{t}M_n(\partial_t)+A}^{-1} = N_n(\partial_t)$
converges to $N(\partial_t)$ in the weak operator topology of $L(L_{2,\rho}(\R;H))$. Thus, it is left to prove that $N(z)$ equals
$(zM(z)+A)^{-1}$ for some $M\in\mathcal{M}_g(H,\nu,\alpha)$.

Let us take an arbitrary $z\in\CC_{\Re>\nu}$.
Again resorting to Propostion \ref{prop:vitaliopval}, we deduce that
\[
    L_n(z)^{-1} \coloneqq (zM_n(z)+A)^{-1} \to N(z)
\]
in the weak operator topology of $L(H)$. Remark \ref{rem:HdomA} implies that $(L_n(z)^{-1})_n$ in fact converges in the weak operator topology of $L(H,\dom(A))$. 
Next, defining 
$$
\beta(z) \coloneqq |z|^2\sup_{n\in \N} \|M_n(z)\|^2<\infty \,,
$$ 
we compute (see also \cite[Proposition 6.2.3 (b)]{STW22})
\[
\Re \langle\psi,(zM_n(z))^{-1} \psi\rangle \geq \frac{\alpha}{\|zM_n(z)\|^2}\|\psi\|^2 \geq \frac{\alpha}{\beta(z)}\|\psi\|^2\quad(\psi\in H)\,,
\]
where we have used that $\Re \langle\phi, zM_n(z)\phi\rangle \geq \alpha\|\phi\|^2$,
which holds since $M_n\in\mathcal{M}_g(H,\nu,\alpha)$.
Thus, substituting $\phi=(zM_n(z))^{-1} \psi$ in the latter inequality we obtain
\[
zM_n(z) \in \mathcal{F}(\alpha,\beta(z)/\alpha;H)
\]for all $n\in \N$. In fact, by the arbitrariness of $z$, the same holds for all $z\in\CC_{\Re>\nu}$. 

Hence, by Theorem  \ref{thm:compFS}, for all $z\in \CC_{\Re>\nu}$ there exists $M(z) \in L(H)$ such that $zM(z) \in \mathcal{F}(\alpha,\beta(z)/\alpha;H)$ and
\[
L_n(z)^{-1}  = (zM_n(z)+A)^{-1} \to N(z)= (zM(z)+A)^{-1}
\]
in the weak operator topology of $L(H,\dom(A))$ for all $z\in \CC_{\Re>\nu}$. Note that $zM(z) \in \mathcal{F}(\alpha,\beta(z)/\alpha;H)$ implies
$\|zM(z)\|\leq \beta(z)/\alpha$, i.e.
\begin{equation}\label{eq:Mestimate}
	\|M(z)\|\leq \frac{\beta(z)}{|z|\alpha} \quad (z\in\CC_{\Re>\nu})\,.
\end{equation}
Hence, $M:\CC_{\Re>\nu}\to L(H)$ is locally bounded. 
Next, we show that $M$ is holomorphic. 
For this, recall that since $z\mapsto L_n(z)^{-1}$ is bounded and holomorphic (see Remark \ref{rem:holo}), we deduce with Corollary \ref{cor:comp} that $N$ is holomorphic. As a consequence (since inversion preserves holomorphicity), $z\mapsto zM(z)+A \in L(\dom(A),H)$ is holomorphic. 
Further on, $z\mapsto zM(z) = zM(z)+A - A$ is scalarly holomorphic in the sense that $z\mapsto \langle \phi,zM(z)\psi\rangle$ is holomorphic for all $\phi\in H$ and $\psi\in \dom(A)$. By the local boundedness of $M$, this implies holomorphicity with values in $L(H)$ as $\dom(A)$ is dense in $H$, by a variant of Dunford's theorem, see, e.g., \cite[Exercise 5.3 (v)]{STW22}. 
Therefore, $M\in\mathcal{M}_g(H,\nu,\alpha)$.

Finally, for the last statement of the theorem, we have shown that
\[
  (\rho M_n(\rho)+A)^{-1} \to    (\rho M(\rho)+A)^{-1}\quad(\rho>\nu)
\]
in the weak operator topology. By Remark \ref{rem:uniqFSC}, $\rho M(\rho) \in L(H)$ is uniquely determined by this convergence. By the identity theorem, this is enough to uniquely identify $M\in \mathcal{M}(H,\nu)$.
\end{proof}

We finally are in the position to prove the compactness statement from the beginning of this section.

\begin{proof}[Proof of Theorem \ref{thm:seqcomp}] By definition, $N_n\colon z\mapsto (zM_n(z)+A)^{-1}$ is a  bounded sequence of material laws. By Theorem \ref{thm:holcomp}, we choose a weakly convergent subsequence in the compact open weak operator topology. Then, the assertion follows from Theorem \ref{thm:mainchar}.
\end{proof}

\begin{proof}[Proof of Corollary \ref{cor:seqcomp_bdd}]
	From the proof of Theorem \ref{thm:mainchar} (see \eqref{eq:Mestimate}) the limit operator $M$, for any $z\in\CC_{\Re>\nu}$, satisfies $\|M(z)\|\leq \frac{|z|}{\alpha}\sup_{n\in\N} \|M_n(z)\|^2$. Thus, applying additional boundedness assumption on $(M_n)_n$ the claim follows. 
\end{proof}

\begin{rem}\label{rem:G-conv_bdd} 
We have commented in Remark \ref{rem:cor} that the result we have for a bounded sequence of material laws (see Corollary \ref{cor:seqcomp_bdd}) is not satisfactory from the point of view of possible applications. 
One can remark the same from the theoretical point of view since the limit operator does not belong (i.e.~we are not able to prove) to the class we started with, as we have with a locally bounded sequences of generalised material laws (see Theorem \ref{thm:seqcomp}).
In this remark we will address a refinement of the second part of Corollary \ref{cor:seqcomp_bdd} for which a family contained in $\mathcal{M}(H,\nu,\alpha)$ 
that is closed under $G$-convergence will be given.

For this purpose 
we introduce  the set (here we take $\nu,\beta,\alpha>0$) $\mathcal{M}(H,\nu,\alpha,\beta)$ of all $M\in\mathcal{M}(H,\nu,\alpha)$ satisfying
\begin{equation}\label{eq:new_cond}
\Re \langle \phi,zM(z) \phi\rangle\geq \frac{1}{\beta|z|}\|zM(z)\phi\|^2
	\quad (z\in\CC_{\Re>\nu}, \, \phi\in H).
\end{equation}
Note that this set is well-defined since $|z|>\nu>0$. Moreover, applying the Cauchy--Schwarz inequality to both inequalities defining $\mathcal{M}(H,\nu,\alpha,\beta)$
we get 
$$
\beta \geq \|M(z)\|\geq \frac{\alpha}{|z|} \quad (z\in \CC_{\Re>\nu}).
$$
Thus, the information that $M$ is bounded given in $\mathcal{M}(H,\nu)$ is redundant, as it is 
also encoded in \eqref{eq:new_cond}.
Additionally, in order to have that conditions are compatible, i.e.~that $\mathcal{M}(H,\nu,\alpha,\beta)$
is not empty, it should be satisfied $\beta\nu\geq\alpha$ .

If in Theorem \ref{thm:mainchar} we start with $M_n\in \mathcal{M}(H,\nu,\alpha,\beta)$, then the statement of the theorem remains true, with $M\in \mathcal{M}(H,\nu,\alpha,\beta)$, and consequently for the compactness: every sequence in $\mathcal{M}(H,\nu,\alpha, \beta)$ has a $G$-convergent subsequence with the limit in $\mathcal{M}(H,\nu,\alpha, \beta)$. Indeed, this is an easy consequence of the fact that $M$ belongs to $\mathcal{M}(H,\nu,\alpha, \beta)$ if and only if for any $z\in\CC_{\Re>\nu}$ the operator $zM(z)$ belongs to $\mathcal{F}(\alpha,\beta|z|,H)$. Since we are interested  in the limit of $\bigl((zM_n(z)+A)^{-1}\bigr)_n$ in the weak operator topology of $L(H,\dom(A))$, under which is $\mathcal{F}(\alpha,\beta|z|,H)$ closed (see Theorem \ref{thm:compFS}), the statement easily follows.

Let us note that any $M\in\mathcal{M}(H,\nu,\alpha)$ satisfying \eqref{eq:Gcomp_addassump} is contained in $\mathcal{M}(H,\nu,\alpha,c^2/\alpha)$.
Thus, the result presented here is indeed a refinement of the second part of
Corollary \ref{cor:seqcomp_bdd}, i.e.~here we have the most general setting 
in which the limit operator is a bounded material law. 
However, condition \eqref{eq:new_cond} still eliminates many interesting examples
(e.g.~material laws constant in $z$ are still not covered).
\end{rem}

We provide some examples illustrating the theory just developed.

\section{Applications}\label{sec:app}

In this section, we focus on some applications. It has been found that $G$-compactness results are particularly interesting in cases, where the operator containing the spatial derivatives lack compactness of the resolvent, see \cite{W12_HO,W14_G} in the context of evolutionary equations, but also see \cite{Tartar2009,Tartar1989,Tartar1990} and the references given there. We particularly refer to the introductory part of \cite{Tartar1989}, where among other things the intricacies of homogenisation of transport equations are highlighted.

\subsection*{A transport equation with longitudinal oscillations}

Here we consider a transport type equation with oscillations in the same direction as the transport occurs. Let $a\in L_\infty(\R)$ with $a\geq \alpha$ for some $\alpha>0$ and consider finding $u\in L_{2,\rho}(\R;L_2(\R))$ (for some $\rho>0$) such that for $f\in C_c(\R\times \R)$ we have
\[
   \partial_t u(t,x)+a(x)\partial_x u(t,x) =a(x) f(t,x)\quad((t,x)\in \R\times \R).
\]
A convenient setting can be found within the framework of evolutionary equations, where $\partial_x = A \colon H^1(\R)\subseteq L_2(\R)\to L_2(\R)$ is the implementation of the spatial derivative and readily confirmed to be skew-selfadjoint. Multiplying by $a(x)^{-1}$ from the left, we have
\[
    (\partial_t a^{-1} + \partial_x) u =  f.
\]
Then it is not difficult to see that Picard's Theorem applies to $A=\partial_x$ and $M(z)=a^{-1}$ for all $z\in \CC$. Hence, $M$ is a material law with $\sbb(M)=-\infty$. 

Next, consider a bounded sequence $(a_n)_n$ in $L_\infty(\R)$ with $a_n\geq \alpha$ for some $\alpha>0$ independently of $n$. Assume that $(a_n^{-1})$ converges to some $b$ in $\sigma(L_\infty,L_1)$, the weak-$*$-topology of $L_\infty(\R)$ (apply \cite[Proposition 13.2.1(c)]{STW22} to obtain convergence of the respective multiplication operators in the weak operator topology). To apply Theorem \ref{thm:mainchar}, it suffices to consider, for $\rho>0$, the equation
\[
     (\rho a_n^{-1}+\partial_x)\phi_n = \psi
\]
for some $\psi\in L_2(\R)$.
It is not difficult to see that 
\[
   \phi_n (x) = \int_{-\infty}^x \e^{-\rho\int_{\xi}^x a_n^{-1}(s) \d s} \psi(\xi)\d \xi.
\]
By Lebesgue's dominated convergence theorem, the pointwise limit of $(\phi_n)_n$ is the weak limit given by
\[
   \phi (x) = \int_{-\infty}^x \e^{-\rho\int_{\xi}^x b (s) \d s} \psi(\xi)\d \xi.
\]
The latter is the solution of the equation
\[
     (\rho b+\partial_x)\phi = \psi.
\]
Hence, $M_n \colon z\mapsto a_n^{-1}$ $G$-converges to $M\colon z\mapsto b$, by Theorem \ref{thm:mainchar}. Note that $M$ is a (standard) material law, i.e.~bounded, despite the fact that none of our results guarantee that.
\begin{rem}
(a) It is not difficult to see that one can also consider $  \partial_t u_n(t,x)+a_n(x)\partial_x u(t,x) =f(t,x)$ instead of the above. A similar computation (albeit Theorem \ref{thm:mainchar} is not directly applicable) confirms that $(u_n)_n$ weakly converges to the unique solution of
\[
     \partial_t u(t,x)+b(x)^{-1}\partial_x u(t,x) =f(t,x).
\]
(b) Note that the formulation presented here does not require fixing an initial value. In fact, the exponential weight in the time-direction somewhat asks implicitly for homogeneous initial values at $-\infty$ (hence the formula for $\partial_{t,\rho}^{-1}$ for $\rho>0$).
\end{rem}

\subsection*{A transport equation with orthogonal oscillations}

This application concerns an equation more related to the one hinted at in \cite{Tartar1989} (see also \cite[Ch. 24]{Tartar2009}). We consider for $f\in C_c(\R\times \R\times \R)$ and $a\in L_\infty(\R)$ an equation of the form
\[
    \partial_t u(t,x,y) + a(y)u(t,x,y) + \partial_x u(t,x,y) = f(t,x,y).
\]
Here, we assume that $\rho>(\|a\|_{\infty}+1)$. Then, Picard's Theorem applies with \[\partial_x = A\colon H^1(\R; L_2(\R)) \subseteq L_2(\R;L_2(\R))\to L_2(\R;L_2(\R))\] (easily seen to be skew-selfadjoint) and $M(z)\coloneqq 1+\frac{1}{z}a$.

Now, we consider a bounded sequence $(a_n)_n$ in $L_\infty(\R)$. Then $(M_n)_n$ is a bounded sequence in $\mathcal{M}(L_2(\R\times \R), \sup_n \|a_n\|_\infty +1, 1)$. By Theorem \ref{thm:seqcomp}, we find a generalised material law $M\in \mathcal{M}_g(L_2(\R\times \R), \sup_n \|a_n\|_\infty +1, 1)$ such that
\[
    (\overline{\partial_t + a_n(y) + \partial_x})^{-1} \to (\overline{\partial_t M(\partial_t)+\partial_x})^{-1}
\]
in the weak operator topology of $L(L_{2,\rho}(\R; L_2(\R \times \R)))$ for $\rho>\sup_n \|a_n\|_\infty +1$.

In the following we shall compute the limit material law, $M(\partial_t)$. For this, note that, by Picard's theorem, the (closure of the) operator $\partial_t+\partial_x$ is continuously invertible in $L_{2,\rho}(\R; L_2(\R \times \R))$ with norm bounded by $1/\rho$ for all $\rho>0$. We shall assume that a subsequence of $(a_n)_n$ (without relabelling) has been chosen in order that, for all $k\in \N$, $a_n^k \to b_k\in L_\infty(\R)$ in the weak-$*$-topology of $L_\infty(\R)$. It follows that $a_n^k\to b_k$ in the weak operator topology as multiplication operators in $L(L_2(\R))$, see also \cite[Proposition 13.2.1(c)]{STW22} (by the tensor-product structure, it follows that $a_n^k\to b_k$ also in the weak operator topology of $L(L_{2,\rho}(\R; L_2(\R \times \R)))$. We assume that $\rho> 4(\sup_n\|a\|_{\infty}+1)\eqqcolon 4\kappa$. Thus, we may now compute the limit expression of 
\[
    \left({(\overline{\partial_t + a_n(y) + \partial_x})}^{-1}\right)_n
\]
in $L(L_{2,\rho}(\R; L_2(\R \times \R)))$. We use a similar strategy as exemplified already in \cite{W12_HO}, see also the more recent \cite[Theorem 13.3.1 and subsequent example]{STW22}. Defining
\[u_n\coloneqq {(\overline{\partial_t + a_n(y) + \partial_x})}^{-1}f\]
as well as $S\coloneqq (\overline{\partial_t + \partial_x})^{-1} \in L(L_{2,\rho}(\R; L_2(\R \times \R)))$. Recall that $\|S\|\leq 1/\rho$. Factoring out $S$ in the equation for $u_n$, we obtain
\[
    u_n  = (1+Sa_n)^{-1} S f = \sum_{k=0}^\infty (-Sa_n)^k Sf = \sum_{k=0}^\infty (-S)^ka_n^k Sf,
\]
where we used a Neumann series argument. Note that for $n\in \N$
\begin{equation}\label{eq:neumann1}
   \Big\|\sum_{k=1}^\infty (-S)^ka_n^k\Big\|\leq \sum_{k=1}^\infty \Big(\frac\kappa\rho\Big)^k \leq \sum_{k=1}^\infty \Big(\frac14\Big)^k = \frac{1}{3/4}-1 = \frac{1}{3}.
\end{equation}
By dominated convergence (interchanging limits and summation), we may deduce that $(u_n)_n$ is weakly convergent to $u$ given by
\[
   u = \sum_{k=0}^\infty (-S)^kb_k Sf = (1+\sum_{k=1}^\infty (-S)^kb_k )Sf.
\]Using a Neumann series argument, which is possible by lower semi-continuity of the norm under weak operator topology convergence and estimate \eqref{eq:neumann1}, we obtain
\[
     \sum_{\ell=0}^\infty \Big(-\sum_{k=1}^\infty (-S)^kb_k \Big)^\ell u = Sf, 
\]and, hence, 
\[
     u -S\sum_{k=1}^\infty (-S)^{k-1}b_k u -S\sum_{\ell=2}^\infty (-S)^{\ell-1}\Big(-\sum_{k=1}^\infty (-S)^{k-1}b_k \Big)^\ell u = Sf.
\]By inspection, it follows that $u\in \ran(S)$ and, thus,
\[
   \overline{( \partial_t +\partial_x)}u  -\sum_{k=1}^\infty (-S)^{k-1}b_k u -\sum_{\ell=2}^\infty (-S)^{\ell-1}\Big(-\sum_{k=1}^\infty (-S)^{k-1}b_k \Big)^\ell u = f.
\]
It follows that
\[
   M_n(\partial_t) = 1+\partial_t^{-1}a_n \stackrel{\text{G}}{\to} \\
   1-\partial_t^{-1}\sum_{k=1}^\infty (-S)^{k-1}b_k u -\partial_t^{-1} \sum_{\ell=2}^\infty (-S)^{\ell-1}\Big(-\sum_{k=1}^\infty (-S)^{k-1}b_k \Big)^\ell.
\]
It is not difficult to see that this limit defines a bounded material law. We emphasise that we confirmed that \emph{both} a memory effect and due to the spatial part also a nonlocal effect after the homogenisation process occur \emph{simultaneously}. 
Indeed, since $S$ is the solution operator of the transport equation in space-time given by the variations of constants formula (i.e.~temporal convolution) of the shift-semigroup (i.e.~spatial non-locality), and since the limit material law contains all iterates of $S$ the non-locality shows up in both temporal as well as spatial variables. 
We conclude with a slight subtlety, our main compactness theorem confirms convergence to a generalised material law for all $\rho>\kappa$; in our derivation however, we needed larger $\rho$ to guarantee existence of the second Neumann series. This implies that the limit material law can actually be holomorphically extended to $\CC_{\Re >\rho}$. 
\section{Conclusion}\label{sec:concl}

We have addressed $G$-compactness for evolutionary equations without additional assumptions on the skew-selfadjoint part $A$. Even though our results also apply to a wider class of material laws in that it is allowed for them to diverge at $\infty$, the price to pay is that the limit generalised material law cannot be shown to be bounded even though the whole material law sequence we started out with was. We provided two applications of our results. These examples as well as the general result itself re-confirm statements made in \cite{W18_NHC} that memory effects are due to a lack of compactness in the equations. However, the spatial operator being non-compact alone does not imply the occurrence of memory effects, oscillations orthogonal to the occurring derivatives do trigger nonlocal effects, though. We conclude the manuscript with an open problem, which so far we did not manage to resolve.
\begin{problem} Let $(M_n)_n$ be a bounded sequence of material laws $G$-converging to some $M$ with respect to $A$ (which in turn does not satisfy any compactness condition).
	
	(a) Prove or disprove that then $M$ is bounded.
	
	(b) Find an easily applicable criterion for $(M_n)_n$, so that $M$ is bounded.
\end{problem}

\section*{Acknowledgments}
M.E.~acknowledges funding by the Croatian Science Foundation under the grant IP-2022-10-7261 (ADESO).
K.B.~acknowledges funding by the Croatian Science Foundation under the grant IP-2022-10-5181 (HOMeoS).
M.W.~gratefully acknowledges the hospitality extended to him during a research visit at the J.J.~Strossmeyer University of Osijek.

\section*{Data availability statement}

We do not analyse or generate any datasets, because our work proceeds within a theoretical
and mathematical approach.

\end{document}